\documentclass[12pt]{article}
\usepackage{amsmath,amssymb,amsthm}
\usepackage{hyperref}
\usepackage{mathrsfs}

\newcommand{\R}{\mathbb{R}}
\newcommand{\T}{\mathbb{T}}
\newcommand{\Z}{\mathbb{Z}}
\newcommand{\dd}{\,\mathrm{d}}

\newtheorem{theorem}{Theorem}[section]

\newtheorem{corollary}[theorem]{Corollary}

\newtheorem{definition}[theorem]{Definition}

\newtheorem{lemma}[theorem]{Lemma}

\theoremstyle{remark}
\newtheorem{remark}[theorem]{Remark}

\theoremstyle{remark}

\begin{document}

\author{Alexandru F. Radu}
\title{Uniqueness in Lorentz Spaces of the 2d Navier-Stokes Equations}
\maketitle

\begin{abstract}
We study uniqueness of mild solutions to the two--dimensional incompressible Navier--Stokes equations on the torus in borderline spatial classes. While Lorentz--space methods yield uniqueness in $C([0,T);L^{2,1}(\T^2))$ via real interpolation and weak $L^2$ control, extending such arguments to larger Lorentz spaces $L^{2,q}$, $1<q<2$, encounters endpoint obstructions. In this paper we prove that uniqueness in $C([0,T);L^{2,q}(\T^2))$ holds provided one assumes a short--time $L^\infty$ smoothing property at every restart time, namely
\[
\lim_{\delta\downarrow 0}\sup_{t\in(T_0,T_0+\delta]}\sqrt{t-T_0}\,\|v(t)\|_{L^\infty(\T^2)}=0
\quad \text{for all } T_0\in[0,T).
\]
The proof combines the restart mild formulation, the $L^1$ bound for the periodic Oseen kernel of $e^{t\Delta}\mathbb{P}\nabla\cdot$, and an explicit Beta--function computation yielding a strict $L^2$ contraction on short intervals. The smoothing assumption is natural in Kato and Koch--Tataru type critical well--posedness frameworks and clarifies how parabolic regularization can replace Lorentz endpoint structure in uniqueness arguments.
\end{abstract}

\section{Introduction}

The incompressible Navier--Stokes equations on the two--dimensional torus $\T^2$ read
\begin{equation}\label{eq:NS}
\left\{
\begin{aligned}
\partial_t v + v\cdot\nabla v - \Delta v + \nabla p &= 0, \\
\nabla\cdot v &= 0,\\
v|_{t=0} &= v_0,
\end{aligned}
\right.
\qquad (t,x)\in (0,T)\times \T^2,
\end{equation}
where $v(t,x)\in\R^2$ is the velocity and $p(t,x)\in\R$ is the pressure. The question of
\emph{uniqueness} of solutions in low regularity spaces is a central theme of the modern
analysis of \eqref{eq:NS}, and is closely intertwined with the program of understanding
criticality, scaling, and the borderline mapping properties of the bilinear term
$B(v,v)=\mathbb{P}\nabla\cdot(v\otimes v)$ (where $\mathbb{P}$ denotes the Leray projector).

Even in two dimensions, where Leray--Hopf weak solutions are globally well--posed under
the classical energy framework, the uniqueness problem becomes delicate when one
works in \emph{critical} spaces and only assumes mild or distributional formulations,
without a priori energy inequalities or additional spacetime integrability. In particular,
a recurring difficulty is that, at the borderline scale, the product $v\otimes w$ may fail
to be well defined (or well controlled) in the function spaces naturally associated to the
parabolic kernel bounds. Lemari\'e--Rieusset has emphasized in several works that uniqueness
in borderline spaces hinges on a subtle interplay between real interpolation, kernel estimates,
and comparison principles; see the monograph \cite{LemarieBook2002} and the more recent
book \cite{LemarieBook2023} for broad perspective, as well as the critical-space uniqueness
paper \cite{LemarieChemin2007} and the Morrey--Campanato viewpoint \cite{LemarieMorrey2007}.

\medskip

\noindent\textbf{Uniqueness at the $L^2$-scale and the Lorentz endpoint.}
A convenient way to encode borderline behavior is to replace Lebesgue spaces by Lorentz spaces.
On $\T^2$, the scaling of \eqref{eq:NS} formally leaves invariant the $L^2$ norm of the velocity.
In that sense, the class $C([0,T);L^2(\T^2))$ is ``critical'' for two--dimensional data.
However, the fixed point / mild solution approach typically requires mapping properties of the
bilinear operator $B(\cdot,\cdot)$ that are \emph{not} available at this endpoint without additional structure.
A refined substitute is the Lorentz space $L^{2,1}$, which is strictly smaller than $L^2$ and enjoys
better product estimates. In particular, a key mechanism in Lemari\'e--Rieusset's uniqueness theory
is the Lorentz estimate
\begin{equation}\label{eq:LR-idea}
\|wz\|_{L^1}\ \lesssim\ \|w\|_{L^{2,1}}\|z\|_{L^{2,\infty}},
\end{equation}
together with the real interpolation identity $L^{2,\infty}=[L^1,L^\infty]_{1/2,\infty}$ and
the corresponding $K$--functional control.
In the two--dimensional periodic setting, these ideas yield uniqueness for mild solutions in
$C([0,T);L^{2,1}(\T^2))$ (and hence in $C([0,T);L^p(\T^2))$ for every $p>2$), an observation that
has been reiterated in various forms in the literature and in expository notes; see, for instance,
\cite{LemarieBook2002,HeNote2024}.

The paper \cite{LemarieRieusset2025HighlySingular} (whose Proposition~1 is one of our starting points) provides a clean and
self-contained implementation of this approach on $\T^2$, relying on periodic kernel bounds for the
Oseen-type operator $\mathrm{e}^{t\Delta}\mathbb{P}\nabla\cdot$ and a decomposition of the difference
of two solutions into pieces controlled in $L^1$ and $L^\infty$, then recombined through Lorentz
interpolation in $L^{2,\infty}$. This is strongly aligned with Lemari\'e--Rieusset's broader program
of exploiting harmonic analysis and real interpolation in critical Navier--Stokes theory
\cite{LemarieBook2002,LemarieMorrey2007}.

\medskip

\noindent\textbf{Motivation for the present work.}
It is natural to ask whether one can extend the uniqueness statement from $L^{2,1}$
to larger Lorentz classes $L^{2,q}$, $1<q<2$ (still strictly contained in $L^2$), or even to
$L^2$ itself. The Lorentz--based strategy behind \eqref{eq:LR-idea} is, however, intrinsically tied
to the endpoint interpolation space $L^{2,\infty}$ and the exponent $\infty$ in the real interpolation
norm; attempting to replace $L^{2,\infty}$ by $L^2=[L^1,L^\infty]_{1/2,2}$ forces an
\emph{integrated} $K$--functional, and the balancing that succeeds for $r=\infty$ generally
produces logarithmic losses for $r=2$. This phenomenon is consistent with the long--standing
attention paid to borderline uniqueness at the $L^2$ scale. In particular, Lemari\'e--Rieusset
has discussed the case $C([0,T];L^2(\T^2))$ as a subtle endpoint problem for mild formulations
in which one does not assume additional spacetime integrability or energy inequalities, and where
the bilinear form may fail to be controlled by naive estimates (see, e.g., \cite{LemarieBook2023} for
context and related discussion).

In parallel, there is a complementary and highly influential viewpoint: rather than working
only with spatial norms at fixed times, one builds solution spaces that encode \emph{parabolic smoothing}.
This includes Kato's strong $L^p$ theory \cite{Kato1984} and, most notably, the critical well-posedness
theorem of Koch and Tataru \cite{KochTataru2001}. Koch--Tataru introduced the space $BMO^{-1}$
for initial data and constructed mild solutions via a fixed point argument in a function space that,
among other components, controls the time-weighted $L^\infty$ norm:
\begin{equation}\label{eq:KT-smoothing}
\sup_{t>0}\ \sqrt{t}\,\|v(t)\|_{L^\infty(\R^d)} < \infty,
\end{equation}
together with additional Carleson/tent-space type seminorms tailored to the heat flow
\cite{KochTataru2001}. Such bounds are natural from the semigroup estimate
$\|e^{t\Delta}f\|_\infty\lesssim t^{-1/2}\|f\|_2$ (in two dimensions) and, in well-posedness classes,
are stable under time translation: restarting the equation at time $T_0$ yields the same estimate
with $t$ replaced by $t-T_0$. Lemari\'e--Rieusset and collaborators have repeatedly explored the
relationship between ``Kato mild'' and ``Leray-type'' solutions in critical spaces (including $BMO^{-1}$),
and how uniqueness principles propagate between them; see, for example, \cite{LemariePrioux2009}
and the survey-style developments in \cite{LemarieBook2002,LemarieBook2023}.

\medskip

\noindent\textbf{Main result and its interpretation.}
The purpose of this paper is to show that one can obtain a clean extension of the
$L^{2,1}$ uniqueness statement to the larger class $L^{2,q}$, $1<q<2$, provided one assumes
a short-time $L^\infty$ smoothing property that is standard in Koch--Tataru/Kato-type frameworks.
More precisely, we assume that for every isnstant $T_0\in[0,T)$ the solution satisfies the
\emph{vanishing} time-weighted bound
\begin{equation}\label{eq:S-intro}
\lim_{\delta\downarrow 0}\ \sup_{t\in (T_0,T_0+\delta]}
\sqrt{t-T_0}\,\|v(t)\|_{L^\infty(\T^2)} = 0.
\end{equation}
Condition \eqref{eq:S-intro} is a natural strengthening of \eqref{eq:KT-smoothing}: it demands not only
critical boundedness of $\sqrt{t-T_0}\|v(t)\|_\infty$, but also its smallness on short intervals after any
restart time. This ``short-time smallness'' is exactly what allows one to close a contraction
argument in $L^2$ without resorting to the endpoint real interpolation machinery, and it is
compatible with time-translation invariant well-posedness classes (including Koch--Tataru).

The main result should be viewed as a ''trade'' between spatial integrability and parabolic smoothing:
Lemari\'e--Rieusset's $L^{2,1}$ theory achieves uniqueness by sharpening the spatial endpoint through
Lorentz structure; our result allows the larger space $L^{2,q}$ by imposing a (very natural) short-time
$L^\infty$ smoothing condition that is built into many fixed-point constructions.

\medskip

\noindent\textbf{Relation to Lemari\'e--Rieusset's program.}
The proof is inspired by Lemari\'e--Rieusset's restart and comparison philosophy
\cite{LemarieBook2002,LemarieChemin2007,LemarieMorrey2007}, and by the periodic kernel framework
developed in \cite{LemarieRieusset2025HighlySingular}. In Proposition~1 of \cite{LemarieRieusset2025HighlySingular} (modeled on Lemari\'e--Rieusset's
classical approach), the difference $w=v_1-v_2$ is decomposed into terms estimated in $L^1$ and $L^\infty$,
with the borderline recombination performed in $L^{2,\infty}$ via real interpolation. Our approach retains
the crucial restart identity and periodic kernel bounds, but changes the closure mechanism:
we use only the $L^1$ kernel estimate
$\|K_{\mathrm{per}}(t)\|_{L^1}\lesssim t^{-1/2}$ and the product bound
$\|v\otimes w\|_{L^2}\le \|v\|_{L^\infty}\|w\|_{L^2}$.
The assumption \eqref{eq:S-intro} supplies the integrable weight
$\|v(t)\|_{L^\infty}\lesssim (t-T_0)^{-1/2}\times(\text{small})$ after a restart at $T_0$.
As a consequence, the time singularity becomes exactly
\[
\int_{T_0}^{t} (t-s)^{-1/2}(s-T_0)^{-1/2}\,\dd s,
\]
which is finite and equal to $\pi$ (a Beta function computation). This yields a strict contraction on
a short interval $(T_0,T_0+\delta]$, and a standard maximal-time continuation argument then forces
$w\equiv 0$ globally.

\medskip

\noindent\textbf{Outline of the proof and organization of the paper.}
Section~\ref{sec:prelim} recalls the periodic mild formulation and kernel estimates for the operator
$\mathrm{e}^{t\Delta}\mathbb{P}\nabla\cdot$, following \cite{LemarieRieusset2025HighlySingular} and the standard Oseen kernel analysis.
We also record the elementary Lorentz embedding $L^{2,q}\hookrightarrow L^2$ for $q\le 2$, which reduces
the uniqueness problem to an $L^2$ contraction argument.
Section~\ref{sec:main} contains the proof of the main theorem.
In the final section we discuss how \eqref{eq:S-intro} fits into existing critical well-posedness theories
and compare the mechanism here with the Lorentz interpolation strategy behind Proposition~1 of \cite{LemarieRieusset2025HighlySingular}
and with the critical $BMO^{-1}$ framework of Koch--Tataru \cite{KochTataru2001}. We also comment on possible
extensions to other settings (e.g.\ whole space, forced equations, or Morrey-type solutions), in the spirit
of Lemari\'e--Rieusset's developments \cite{LemarieMorrey2007,LemariePrioux2009,LemarieBook2023}.

\medskip

\noindent\textbf{Acknowledgements.} A.R. was partially supported by a grant of the Ministry of Research, Innovation and Digitization, CCCDI - UEFISCDI, project number ROSUA-2024-0001, within PNCDI IV.  The author wishes to thank Liviu Ignat for his insightful advice for the preparation of this paper.

\bigskip

\section{Preliminaries}\label{sec:prelim}

\subsection{Notation and basic functional-analytic setup}

We write $\T^2:=\R^2/\Z^2$ for the two--dimensional flat torus. All function spaces below
are over $\T^2$ unless otherwise stated. For $1\le p\le\infty$, $\|\cdot\|_{p}$ denotes the
$L^p(\T^2)$ norm. The symbol $\lesssim$ means an inequality up to a universal constant
that may depend on fixed parameters such as the dimension (here $2$), but not on the time variables
under consideration.

We denote by $\mathbb P$ the Leray projector onto divergence--free vector fields on $\T^2$,
defined Fourier--multiplicatively by
\[
\widehat{\mathbb P f}(k) = \left(I - \frac{k\otimes k}{|k|^2}\right)\widehat f(k),
\qquad k\in\Z^2\setminus\{0\},
\]
and $\widehat{\mathbb P f}(0)=\widehat f(0)$. We will not need to impose any normalization on the mean;
the projector is bounded on $L^p$ for $1<p<\infty$ and acts as a Calder\'on--Zygmund operator.

The Navier--Stokes equations on $\T^2$ are
\begin{equation}\label{eq:NS-torus}
\left\{
\begin{aligned}
\partial_t v - \Delta v + \mathbb P\nabla\cdot (v\otimes v) &= 0,\\
\nabla\cdot v &= 0,\\
v|_{t=0} &= v_0.
\end{aligned}
\right.
\end{equation}
Here $v=v(t,x)\in\R^2$ is the velocity and $v\otimes v$ is the $2\times 2$ matrix with entries
$(v\otimes v)_{ij}=v_i v_j$.

\subsection{Periodic mild formulation and time restart}

Let $(e^{t\Delta})_{t\ge 0}$ denote the heat semigroup on $\T^2$. We recall that
$e^{t\Delta}$ is a contraction on all $L^p$ spaces, and it enjoys the standard smoothing estimate
\begin{equation}\label{eq:heat-L2-Linfty}
\|e^{t\Delta}f\|_{\infty}\ \lesssim\ t^{-1/2}\|f\|_{2},\qquad 0<t\le 1,
\end{equation}
as well as the analogous $L^p\to L^r$ estimates.

We say that $v$ is a \emph{mild solution} of \eqref{eq:NS-torus} on $[0,T)$ if $v\in C([0,T);L^2(\T^2))$
is divergence free for each $t$ (in the distributional sense) and satisfies the integral equation
\begin{equation}\label{eq:mild-0}
v(t)=e^{t\Delta}v_0 - \int_0^t e^{(t-s)\Delta}\,\mathbb P\nabla\cdot\bigl(v(s)\otimes v(s)\bigr)\,\dd s,
\qquad 0\le t<T.
\end{equation}
Equivalently, one may rewrite the Duhamel term as a spatial convolution with a time--dependent kernel
$K_{\mathrm{per}}(t,\cdot)$, the periodic Oseen-type kernel associated with the operator
$e^{t\Delta}\mathbb P\nabla\cdot$. Namely, for each fixed $t>0$, $K_{\mathrm{per}}(t,\cdot)$ is a
$2\times 2\times 2$ tensor-valued kernel such that
\begin{equation}\label{eq:kernel-form}
e^{t\Delta}\mathbb P\nabla\cdot F = K_{\mathrm{per}}(t,\cdot) * F
\qquad \text{for all suitable matrix fields }F,
\end{equation}
where $*$ denotes convolution on $\T^2$ and the contraction in indices is the natural one
(so the output is a vector field).
We will only use norm bounds on $K_{\mathrm{per}}$ and do not need an explicit formula.

\medskip

\noindent\textbf{Time restart.}
A key feature of \eqref{eq:mild-0} is that it can be restarted at any intermediate time.
Fix $T_0\in[0,T)$ and set $v(T_0)\in L^2(\T^2)$.
Using the semigroup property $e^{t\Delta}=e^{(t-T_0)\Delta}e^{T_0\Delta}$ and splitting the Duhamel integral,
one checks directly that \eqref{eq:mild-0} is equivalent to the following \emph{restart formulation}:
for all $t\in[T_0,T)$,
\begin{equation}\label{eq:mild-restart-prelim}
v(t)=e^{(t-T_0)\Delta}v(T_0) - \int_{T_0}^t e^{(t-s)\Delta}\,\mathbb P\nabla\cdot\bigl(v(s)\otimes v(s)\bigr)\,\dd s.
\end{equation}
Equivalently, in kernel form,
\begin{equation}\label{eq:mild-restart-kernel}
v(t)=e^{(t-T_0)\Delta}v(T_0) - \int_{T_0}^t K_{\mathrm{per}}(t-s,\cdot) * \bigl(v(s)\otimes v(s)\bigr)\,\dd s.
\end{equation}
Identity \eqref{eq:mild-restart-prelim}--\eqref{eq:mild-restart-kernel} is the starting point of
the maximal-time coincidence argument used in Lemari\'e--Rieusset's approach and in \cite{LemarieRieusset2025HighlySingular}.

\subsection{Kernel bounds for $e^{t\Delta}\mathbb P\nabla\cdot$ on $\T^2$}\label{subsec:kernel}

We record the periodic kernel estimates that we will use throughout. The following bounds are
standard and appear (in essentially this form) in \cite{LemarieRieusset2025HighlySingular}; they may be proved by starting
from the corresponding whole-space Oseen kernel on $\R^2$ and then periodizing, using the decay of the
whole-space kernel away from the origin.

\begin{lemma}[Periodic kernel bounds]\label{lem:Kper-bounds}
Let $K_{\mathrm{per}}(t,\cdot)$ be the periodic kernel for the operator $e^{t\Delta}\mathbb P\nabla\cdot$
on $\T^2$ as in \eqref{eq:kernel-form}. Then there exists a constant $C>0$ such that for all $t>0$,
\begin{equation}\label{eq:Kper-L1}
\|K_{\mathrm{per}}(t,\cdot)\|_{L^1(\T^2)}\le C\,t^{-1/2},
\end{equation}
and
\begin{equation}\label{eq:Kper-Linfty}
\|K_{\mathrm{per}}(t,\cdot)\|_{L^\infty(\T^2)}\le C\,(1+t^{-3/2}).
\end{equation}
\end{lemma}

\begin{proof}[Proof sketch]
The proof is classical; we briefly indicate the mechanism and refer to \cite{LemarieRieusset2025HighlySingular} for a detailed
presentation in the periodic setting.

On $\R^2$, the kernel of $e^{t\Delta}\mathbb P\nabla\cdot$ is the (tensor-valued) Oseen kernel
$K(t,x)$, which satisfies the pointwise bound (for $t>0$)
\begin{equation}\label{eq:Oseen-pointwise}
|K(t,x)|\ \lesssim\ \frac{1}{(\sqrt{t}+|x|)^{3}},
\end{equation}
reflecting one spatial derivative of the heat kernel combined with a Calder\'on--Zygmund projection.
From \eqref{eq:Oseen-pointwise} one obtains
\[
\|K(t,\cdot)\|_{L^1(\R^2)}\lesssim t^{-1/2},
\qquad
\|K(t,\cdot)\|_{L^\infty(\R^2)}\lesssim t^{-3/2}.
\]
The periodic kernel is obtained by periodization
\[K_{\mathrm{per}}(t,x)=\sum_{n\in\Z^2}K(t,x+n)\] (in the sense of distributions).
Using \eqref{eq:Oseen-pointwise}, the $n=0$ term yields the same singular behavior as $t\downarrow 0$,
while the terms $n\neq 0$ contribute uniformly bounded tails because $|x+n|\gtrsim 1$ for $x\in\T^2$.
This yields \eqref{eq:Kper-L1} and \eqref{eq:Kper-Linfty} with the extra harmless ``$+1$'' in
\eqref{eq:Kper-Linfty} coming from the $n\neq 0$ tail. For a detailed periodic derivation and constants,
see \cite{LemarieRieusset2025HighlySingular}, Proposition~1 and the kernel bounds cited therein.
\end{proof}

\medskip

We emphasize that our main uniqueness proof uses only \eqref{eq:Kper-L1} and the Young inequality
\begin{equation}\label{eq:Young-L1-L2}
\|K_{\mathrm{per}}(t-s,\cdot)*F\|_{L^2}\ \le\ \|K_{\mathrm{per}}(t-s,\cdot)\|_{L^1}\,\|F\|_{L^2}.
\end{equation}
The stronger $L^\infty$ bound \eqref{eq:Kper-Linfty} is essential in the Lorentz endpoint arguments
of Lemari\'e--Rieusset and \cite{LemarieRieusset2025HighlySingular}, but will not enter our $L^2$ contraction closure.

\subsection{Lorentz spaces and the embedding $L^{2,q}\hookrightarrow L^2$}\label{subsec:lorentz}

We briefly recall the Lorentz spaces $L^{p,r}(\T^2)$ and record the embedding needed to reduce our
uniqueness statement to an $L^2$ argument.

\medskip

\noindent\textbf{Definition.}
For a measurable function $f$ on $\T^2$, let $f^\ast$ denote its decreasing rearrangement. For
$1\le p<\infty$ and $1\le r\le\infty$, the Lorentz norm is defined by
\[
\|f\|_{L^{p,r}}
:=
\begin{cases}
\displaystyle
\left(\int_0^{|\T^2|}\bigl[t^{1/p}f^\ast(t)\bigr]^r\,\frac{\dd t}{t}\right)^{1/r}, & 1\le r<\infty,\\[2.0ex]
\displaystyle
\sup_{0<t<|\T^2|}\ t^{1/p}f^\ast(t), & r=\infty.
\end{cases}
\]
We recall the monotonicity in the second index: for fixed $p$, if $1\le r_1\le r_2\le\infty$, then
\begin{equation}\label{eq:lorentz-monotone}
L^{p,r_1}(\T^2)\hookrightarrow L^{p,r_2}(\T^2)
\qquad\text{and}\qquad
\|f\|_{L^{p,r_2}}\le C\,\|f\|_{L^{p,r_1}}.
\end{equation}
In particular, $L^{p,p}=L^p$.

\begin{lemma}[Embedding $L^{2,q}\hookrightarrow L^2$ for $q\le 2$]\label{lem:L2q-into-L2}
Let $1\le q\le 2$. Then
\[
L^{2,q}(\T^2)\hookrightarrow L^2(\T^2)
\]
continuously. In particular, if $v\in C([0,T),L^{2,q})$ with $q\le 2$, then $v\in C([0,T),L^2)$.
\end{lemma}

\begin{proof}
Since $L^{2,2}=L^2$, the claim follows immediately from \eqref{eq:lorentz-monotone} with $p=2$,
$r_1=q$, $r_2=2$:
\[
\|f\|_{L^2}=\|f\|_{L^{2,2}}\le C\,\|f\|_{L^{2,q}}.
\]
Continuity in time is preserved by continuous embeddings: if $t\mapsto v(t)$ is continuous in
$L^{2,q}$ and the embedding into $L^2$ is continuous, then $t\mapsto v(t)$ is continuous in $L^2$.
\end{proof}

\subsection{A basic $L^2$ product inequality used in the main proof}

We record the elementary product estimate that will be used to place the Navier--Stokes nonlinearity
in $L^2$ once an $L^\infty$ bound is available.

\begin{lemma}\label{lem:Linfty-L2}
For $f\in L^\infty(\T^2)$ and $g\in L^2(\T^2)$, one has $fg\in L^2(\T^2)$ and
\[
\|fg\|_{L^2}\le \|f\|_{L^\infty}\|g\|_{L^2}.
\]
Consequently, for vector fields $a,b$,
\[
\|a\otimes b\|_{L^2}\le \|a\|_{L^\infty}\|b\|_{L^2}.
\]
\end{lemma}

\begin{proof}
This is immediate from the definition of $L^\infty$ and Cauchy--Schwarz:
\[
\|fg\|_2^2=\int_{\T^2}|f|^2|g|^2\le \|f\|_\infty^2\int_{\T^2}|g|^2=\|f\|_\infty^2\|g\|_2^2.
\]
The tensor estimate follows componentwise.
\end{proof}

\medskip

The combination of Lemma~\ref{lem:Kper-bounds} (in particular \eqref{eq:Kper-L1}),
Lemma~\ref{lem:L2q-into-L2}, and Lemma~\ref{lem:Linfty-L2} will be the analytic backbone of the
$L^2$ contraction argument in Section~\ref{sec:main}. The role of the short-time smoothing assumption
\eqref{eq:S} is to ensure that $\|v(t)\|_{L^\infty}$ behaves like $(t-T_0)^{-1/2}$ with a \emph{small}
prefactor after any restart time $T_0$, which renders the Volterra kernel
$(t-s)^{-1/2}(s-T_0)^{-1/2}$ integrable and yields a strict contraction after a sufficiently short restart.

\section{Main Result}\label{sec:main}

\begin{theorem}\label{main}
Let $T\in(0,\infty]$ and let $1<q<2$. Let $v_1,v_2$ be mild solutions of the periodic Navier--Stokes equation on
$[0,T)\times\mathbb T^2$ with the same initial data
\[
v_1(0)=v_2(0)\in L^{2,q}(\mathbb T^2).
\]
Assume:
\begin{enumerate}
\item $v_j\in C([0,T),L^{2,q}(\mathbb T^2))$ for $j=1,2$.
\item (Short-time smoothing at every restart time) For every $T_0\in[0,T)$,
\begin{equation}\label{eq:S}
\lim_{\delta\downarrow 0}\ \sup_{t\in(T_0,T_0+\delta]}
\sqrt{t-T_0}\,\|v_j(t)\|_{L^\infty(\mathbb T^2)}=0,\qquad j=1,2.
\end{equation}
\end{enumerate}
Then $v_1\equiv v_2$ on $[0,T)$.
\end{theorem}

\begin{proof}
\textbf{Step 1.}
Since $1<q<2$, the Lorentz embedding monotonicity in the second index yields
\[
L^{2,q}(\mathbb T^2)\hookrightarrow L^{2,2}(\mathbb T^2)=L^2(\mathbb T^2)
\]
continuously. Hence $v_j\in C([0,T),L^{2,q})$ implies $v_j\in C([0,T),L^2)$, and therefore
\[
w:=v_1-v_2\in C([0,T),L^2(\mathbb T^2)).
\]

\textbf{Step 2.}
Define
\[
T^*:=\sup\{S\in[0,T): w(t)=0\ \text{in }L^2 \text{ for all }t\in[0,S)\}.
\]
If $T^*=T$ we are done. Assume $T^*<T$. Since $w\in C([0,T),L^2)$ and $w(t)=0$ for $t<T^*$, we have
\begin{equation}\label{eq:wTstar}
w(T^*)=0\quad\text{in }L^2(\mathbb T^2).
\end{equation}

\textbf{Step 3.}
By the time-restart property of mild solutions, for each $j=1,2$ and all $t\in[T^*,T)$,
\begin{equation}\label{eq:mild-restart}
v_j(t)=e^{(t-T^*)\Delta}v_j(T^*)
-\int_{T^*}^{t} K_{\mathrm{per}}(t-s)*\bigl(v_j(s)\otimes v_j(s)\bigr)\,ds,
\end{equation}
where $K_{\mathrm{per}}$ is the periodic Oseen-type kernel associated to the Leray projector. Subtracting the identities \eqref{eq:mild-restart} for $j=1$ and $j=2$,
the linear terms cancel because of \eqref{eq:wTstar}, hence
\begin{equation}\label{eq:w-duhamel}
w(t)=-\int_{T^*}^{t} K_{\mathrm{per}}(t-s)*
\Bigl(v_1(s)\otimes v_1(s)-v_2(s)\otimes v_2(s)\Bigr)\,ds.
\end{equation}
Using the algebraic identity
\[
v_1\otimes v_1-v_2\otimes v_2=(v_1-v_2)\otimes v_1+v_2\otimes(v_1-v_2)=w\otimes v_1+v_2\otimes w,
\]
we rewrite \eqref{eq:w-duhamel} as
\begin{equation}\label{eq:w-duhamel2}
w(t)=-\int_{T^*}^{t} K_{\mathrm{per}}(t-s)*
\Bigl(w(s)\otimes v_1(s)+v_2(s)\otimes w(s)\Bigr)\,ds.
\end{equation}

\textbf{Step 4.}
We use the kernel estimate
\begin{equation}\label{eq:K-L1}
\|K_{\mathrm{per}}(t)\|_{L^1(\mathbb T^2)}\le C\,t^{-1/2},\qquad t>0.
\end{equation}
Fix $t\in[T^*,T)$. Taking $L^2$ norms in \eqref{eq:w-duhamel2} and using the triangle inequality gives
\begin{align}
\|w(t)\|_{2}
&\le \int_{T^*}^{t}\bigl\|K_{\mathrm{per}}(t-s)*\bigl(w(s)\otimes v_1(s)\bigr)\bigr\|_{2}\,ds
\nonumber\\
&\quad +\int_{T^*}^{t}\bigl\|K_{\mathrm{per}}(t-s)*\bigl(v_2(s)\otimes w(s)\bigr)\bigr\|_{2}\,ds.
\label{eq:basic-tri}
\end{align}
By Young's inequality $L^1*L^2\to L^2$ and \eqref{eq:K-L1},
\[
\|K_{\mathrm{per}}(t-s)*F\|_2\le \|K_{\mathrm{per}}(t-s)\|_1\|F\|_2\le C(t-s)^{-1/2}\|F\|_2,
\]
hence from \eqref{eq:basic-tri}
\begin{align}
\|w(t)\|_2
&\le C\int_{T^*}^{t}(t-s)^{-1/2}\,\|w(s)\otimes v_1(s)\|_2\,ds
+ C\int_{T^*}^{t}(t-s)^{-1/2}\,\|v_2(s)\otimes w(s)\|_2\,ds.
\label{eq:after-young}
\end{align}
Using the product estimate $L^\infty\times L^2\subset L^2$,
\[
\|w(s)\otimes v_1(s)\|_2\le \|v_1(s)\|_\infty\|w(s)\|_2,\qquad
\|v_2(s)\otimes w(s)\|_2\le \|v_2(s)\|_\infty\|w(s)\|_2,
\]
we obtain
\begin{equation}\label{eq:Volterra}
\|w(t)\|_2\le C\int_{T^*}^{t}(t-s)^{-1/2}\bigl(\|v_1(s)\|_\infty+\|v_2(s)\|_\infty\bigr)\|w(s)\|_2\,ds.
\end{equation}

\textbf{Step 5.}
Fix $\delta\in(0,\min\{1,T-T^*\})$ and restrict to $t\in[T^*,T^*+\delta]$. Define
\begin{equation}\label{eq:Mdelta}
M(\delta):=\sup_{s\in(T^*,T^*+\delta]}\sqrt{s-T^*}\,\bigl(\|v_1(s)\|_\infty+\|v_2(s)\|_\infty\bigr).
\end{equation}
Then for $s\in(T^*,T^*+\delta]$,
\begin{equation}\label{eq:bound-vinf}
\|v_1(s)\|_\infty+\|v_2(s)\|_\infty\le \frac{M(\delta)}{\sqrt{s-T^*}}.
\end{equation}
By the assumption \eqref{eq:S} applied at $T_0=T^*$ for each $v_j$, we have
\begin{equation}\label{eq:Mto0}
\lim_{\delta\downarrow 0} M(\delta)=0.
\end{equation}
Next, set
\[
\|w\|_{L^2_\delta}:=\sup_{t\in[T^*,T^*+\delta]}\|w(t)\|_2.
\]
Using \eqref{eq:Volterra}, \eqref{eq:bound-vinf}, and the definition of $\|w\|_{L^2_\delta}$ yields, for $t\in[T^*,T^*+\delta]$,
\begin{equation}\label{eq:w-est-preBeta}
\|w(t)\|_2
\le C\,M(\delta)\,\|w\|_{L^2_\delta}\int_{T^*}^{t}(t-s)^{-1/2}(s-T^*)^{-1/2}\,ds.
\end{equation}
Compute the integral explicitly. Let $a:=t-T^*>0$ and set $s=T^*+a\theta$, $ds=a\,d\theta$, $\theta\in[0,1]$:
\begin{align*}
\int_{T^*}^{t}(t-s)^{-1/2}(s-T^*)^{-1/2}\,ds
&=\int_0^1 (a(1-\theta))^{-1/2}(a\theta)^{-1/2}\,a\,d\theta\\
&=\int_0^1 \theta^{-1/2}(1-\theta)^{-1/2}\,d\theta\\
&=\mathrm{B}\Bigl(\frac12,\frac12\Bigr)=\pi.
\end{align*}
Hence \eqref{eq:w-est-preBeta} becomes
\[
\|w(t)\|_2\le C\pi\,M(\delta)\,\|w\|_{L^2_\delta}\qquad\forall t\in[T^*,T^*+\delta].
\]
Taking the supremum over $t\in[T^*,T^*+\delta]$ yields
\begin{equation}\label{eq:contraction}
\|w\|_{L^2_\delta}\le C\pi\,M(\delta)\,\|w\|_{L^2_\delta}.
\end{equation}

\textbf{Step 6.}
By \eqref{eq:Mto0}, choose $\delta>0$ so small that $C\pi\,M(\delta)<1$. Then \eqref{eq:contraction} implies
$\|w\|_{L^2_\delta}=0$, i.e.\ $w(t)=0$ in $L^2$ for all $t\in[T^*,T^*+\delta]$.
This contradicts the definition of $T^*$ unless $T^*=T$. Therefore $T^*=T$ and $v_1\equiv v_2$ on $[0,T)$.
\end{proof}

\section{Applications: consequences of the main result.}
\label{sec:cor}

In this section we record two corollaries of Theorem~\ref{main}.
The first one is a ``bridge'' statement connecting Lorentz-data solutions to
Kato/Koch--Tataru type path spaces. The second one is a quantitative
short-time stability estimate in $L^2$ (in particular, a Lipschitz-type
dependence on restart data).

\subsection{Uniqueness among Koch--Tataru/Kato-type solutions}

We begin by isolating the analytic feature of Koch--Tataru's framework
that is relevant for our uniqueness argument. In \cite{KochTataru2001},
Koch and Tataru construct mild solutions in $\R^d$ for small data in
$BMO^{-1}$ in a solution space that, in particular, controls the time-weighted
$L^\infty$ norm $\sup_{t>0}\sqrt{t}\|u(t)\|_\infty$ (together with a tent-space
Carleson-type norm). A key point is that their solution space is
time-translation invariant, so the same type of control is available after
restarting at any time $T_0$. (See also \cite{AuscherFrey1310}
for an operator-theoretic proof of Koch--Tataru's result.)

On the torus $\T^2$ one can formulate an analogous ``Kato/Koch--Tataru type''
path-space hypothesis that is precisely designed to imply the short-time
smoothing condition \eqref{eq:S}. For our purposes it is enough to work with
the following abstract definition.

\begin{definition}[A restart-smoothing path class]\label{def:X0}
Let $T\in(0,\infty]$. We say that a mild solution $v$ on $[0,T)\times\T^2$
belongs to the class $\mathcal{X}_0([0,T))$ if for every restart time $T_0\in[0,T)$,
\begin{equation}\label{eq:X0-def}
\lim_{\delta\downarrow 0}\ \sup_{t\in(T_0,T_0+\delta]}\sqrt{t-T_0}\,\|v(t)\|_{L^\infty(\T^2)}=0.
\end{equation}
\end{definition}

Of course, Definition~\ref{def:X0} is identical to \eqref{eq:S} (and is therefore
tailored to ~\ref{main}). The point is that many
standard critical mild-solution constructions satisfy \eqref{eq:X0-def}:
for instance, in the Koch--Tataru framework, the ``vanishing'' version of the
critical theory (often associated with $VMO^{-1}$ data) yields the smallness
$\lim_{t\downarrow 0}\sqrt{t}\|u(t)\|_\infty=0$ as part of the topology of the solution
space, and time-translation invariance then yields the restarted version at all $T_0$
(see \cite{KochTataru2001} and the discussion around the vanishing subspace in later
expositions).

\begin{corollary}[Uniqueness in $C_tL_x^{2,q}$ within $\mathcal{X}_0$]\label{cor:A}
Let $T\in(0,\infty]$ and let $1<q<2$. Let $v_1,v_2$ be mild solutions on $[0,T)\times\T^2$
such that
\[
v_1,v_2\in C([0,T),L^{2,q}(\T^2))\cap \mathcal{X}_0([0,T)).
\]
If $v_1(0)=v_2(0)$ in $L^{2,q}(\T^2)$, then $v_1\equiv v_2$ on $[0,T)$.
\end{corollary}

\begin{proof}
Since $q<2$, we have the continuous embedding $L^{2,q}(\T^2)\hookrightarrow L^2(\T^2)$,
hence $v_j\in C([0,T),L^2)$ and the difference $w=v_1-v_2$ belongs to $C([0,T),L^2)$.
Moreover, $v_1,v_2\in \mathcal{X}_0([0,T))$ means that each $v_j$ satisfies the short-time
restart smoothing condition \eqref{eq:S} (with $T_0$ arbitrary). Therefore, all hypotheses
of Theorem~\ref{main} are satisfied, and we conclude $v_1\equiv v_2$.
\end{proof}

\begin{remark}[How this interfaces with Koch--Tataru]\label{rem:KT-interface}
Corollary~\ref{cor:A} can be viewed as a bridge statement:
if a mild solution belongs to a Koch--Tataru/Kato-type path space that implies
the restart-vanishing estimate \eqref{eq:X0-def} (for instance, the vanishing
subclass of the Koch--Tataru space associated with $VMO^{-1}$ data),
then it is unique among all mild solutions in $C([0,T),L^{2,q})$ that satisfy the same
restart smoothing. In this sense, short-time parabolic smoothing replaces the Lorentz
endpoint structure ($L^{2,1}$ vs.\ $L^{2,\infty}$) used in the classical Lemari\'e--Rieusset
argument.
\end{remark}

\subsection{Short-time stability in $L^2$}

We now record a quantitative estimate that is implicit in the contraction proof.
It yields a Lipschitz-type dependence of the solution on restart data in the $L^2$ topology
over short intervals, with an explicit stability factor.

\begin{corollary}[Short-time $L^2$ stability under restart smoothing]\label{cor:B}
Let $T\in(0,\infty]$ and let $1<q<2$.
Let $v_1,v_2$ be mild solutions on $[0,T)\times\T^2$ satisfying
\[
v_1,v_2\in C([0,T),L^{2,q}(\T^2)),
\]
and assume that for some fixed restart time $T_0\in[0,T)$ the following smallness holds:
\begin{equation}\label{eq:Md-small}
M_{T_0}(\delta):=\sup_{t\in(T_0,T_0+\delta]}
\sqrt{t-T_0}\,\bigl(\|v_1(t)\|_{L^\infty}+\|v_2(t)\|_{L^\infty}\bigr)
\ \le\ m,
\end{equation}
for some $\delta\in(0,\min\{1,T-T_0\})$ and some $m\ge 0$.

Then the difference $w=v_1-v_2$ satisfies, on $[T_0,T_0+\delta]$,
\begin{equation}\label{eq:stability-est}
\sup_{t\in[T_0,T_0+\delta]}\|w(t)\|_{L^2}
\ \le\
\frac{\|w(T_0)\|_{L^2}}{1-\kappa},
\qquad
\kappa:=C\pi\, m,
\end{equation}
where $C$ is the constant from the kernel bound $\|K_{\mathrm{per}}(t)\|_{L^1}\le C t^{-1/2}$.
In particular, if $\kappa<1$ then $v_1\equiv v_2$ on $[T_0,T_0+\delta]$ whenever $w(T_0)=0$.
\end{corollary}

\begin{proof}
\textbf{Step 1.}
As in the proof of Theorem~\ref{main}, since $q<2$ we have
$L^{2,q}\hookrightarrow L^2$ and thus $v_j\in C([0,T),L^2)$, so $w:=v_1-v_2\in C([0,T),L^2)$.
Restart the mild formulation at time $T_0$:
\[
v_j(t)=e^{(t-T_0)\Delta}v_j(T_0)
-\int_{T_0}^{t} K_{\mathrm{per}}(t-s)*\bigl(v_j(s)\otimes v_j(s)\bigr)\,ds,
\qquad t\in[T_0,T).
\]
Subtracting yields
\begin{equation}\label{eq:w-restart-T0}
w(t)=e^{(t-T_0)\Delta}w(T_0)
-\int_{T_0}^{t} K_{\mathrm{per}}(t-s)*
\Bigl(w(s)\otimes v_1(s)+v_2(s)\otimes w(s)\Bigr)\,ds.
\end{equation}

\textbf{Step 2.}
Take $L^2$ norms in \eqref{eq:w-restart-T0}. Using that the heat semigroup is a contraction on $L^2$,
\[
\|e^{(t-T_0)\Delta}w(T_0)\|_2\le \|w(T_0)\|_2,
\]
and using Young's inequality $L^1*L^2\to L^2$ together with $\|K_{\mathrm{per}}(t)\|_{1}\le C t^{-1/2}$,
we obtain, for $t\in[T_0,T_0+\delta]$,
\begin{align}
\|w(t)\|_2
&\le \|w(T_0)\|_2
+ C\int_{T_0}^{t}(t-s)^{-1/2}\,\|w(s)\otimes v_1(s)\|_2\,ds
+ C\int_{T_0}^{t}(t-s)^{-1/2}\,\|v_2(s)\otimes w(s)\|_2\,ds.
\label{eq:w-stab-1}
\end{align}
Using $\|fg\|_2\le \|f\|_\infty\|g\|_2$,
\[
\|w(s)\otimes v_1(s)\|_2\le \|v_1(s)\|_\infty\|w(s)\|_2,\qquad
\|v_2(s)\otimes w(s)\|_2\le \|v_2(s)\|_\infty\|w(s)\|_2,
\]
so \eqref{eq:w-stab-1} becomes
\begin{equation}\label{eq:w-stab-2}
\|w(t)\|_2
\le \|w(T_0)\|_2
+ C\int_{T_0}^{t}(t-s)^{-1/2}\bigl(\|v_1(s)\|_\infty+\|v_2(s)\|_\infty\bigr)\|w(s)\|_2\,ds.
\end{equation}

\textbf{Step 3.}
From \eqref{eq:Md-small}, for $s\in(T_0,T_0+\delta]$ we have
\[
\|v_1(s)\|_\infty+\|v_2(s)\|_\infty
\le \frac{m}{\sqrt{s-T_0}}.
\]
Define
\[
\|w\|_{L^2_\delta}:=\sup_{t\in[T_0,T_0+\delta]}\|w(t)\|_2.
\]
Then for $t\in[T_0,T_0+\delta]$, \eqref{eq:w-stab-2} implies
\begin{equation}\label{eq:w-stab-3}
\|w(t)\|_2
\le \|w(T_0)\|_2
+ C m\,\|w\|_{L^2_\delta}\int_{T_0}^{t}(t-s)^{-1/2}(s-T_0)^{-1/2}\,ds.
\end{equation}
Compute the time integral exactly: with $a=t-T_0>0$ and $s=T_0+a\theta$,
\[
\int_{T_0}^{t}(t-s)^{-1/2}(s-T_0)^{-1/2}\,ds
=\int_0^1 \theta^{-1/2}(1-\theta)^{-1/2}\,d\theta
=\mathrm{B}\Bigl(\frac12,\frac12\Bigr)=\pi.
\]
Thus \eqref{eq:w-stab-3} yields
\[
\|w(t)\|_2\le \|w(T_0)\|_2 + C\pi m\,\|w\|_{L^2_\delta},
\qquad t\in[T_0,T_0+\delta].
\]
Taking the supremum over $t\in[T_0,T_0+\delta]$ gives
\[
\|w\|_{L^2_\delta}\le \|w(T_0)\|_2 + C\pi m\,\|w\|_{L^2_\delta}.
\]
Rearranging,
\[
(1-\kappa)\|w\|_{L^2_\delta}\le \|w(T_0)\|_2,\qquad \kappa:=C\pi m.
\]
If $\kappa<1$ this implies \eqref{eq:stability-est}. Finally, if $w(T_0)=0$ then
$\|w\|_{L^2_\delta}=0$, hence $v_1\equiv v_2$ on $[T_0,T_0+\delta]$.
\end{proof}

\begin{remark}
If $v_1,v_2$ satisfy the restart-vanishing smoothing condition \eqref{eq:S} at time $T_0$,
then $M_{T_0}(\delta)\to 0$ as $\delta\downarrow 0$. In particular, one may choose $\delta$
so that $\kappa<1$ in \eqref{eq:stability-est}. Setting $T_0=T^*$ (the maximal coincidence time)
and $w(T^*)=0$, Corollary~\ref{cor:B} reproduces the key contraction step in
Theorem~\ref{main}.
\end{remark}

\section{Conflict of Interest Statement}
The author has no conflicts of interest to declare.

\section{Declaration of generative AI and AI-assisted technologies in the manuscript preparation process}
The author did not utilize generative AI and AI-assisted technologies in the manuscript preparation process.

\noindent\textsc{"Simion Stoilow" Institute of Mathematics of the Romanian Academy, Calea Grivitei Street, no. 21, 010702 Bucharest, Romania.}
\newline\noindent\textit{Email address}: \href{mailto:sasharadu@icloud.com}{sasharadu@icloud.com}.

\end{document}